\documentclass{amsart}

\title[$({\mathbf Q}, +)$ is not automatic]{The additive group of the rationals does not have an automatic presentation}
\author{Todor Tsankov}
\date{April 2009}
\address{Analyse fonctionnelle, bo{\^ \i}te 186 \\ Universit{\' e} Paris 6 \\
4 Place Jussieu \\ 75252 Paris CEDEX 05 \\ France}
\email{todor@math.jussieu.fr}
\keywords{automatic structures, FA-presentable, abelian groups, additive combinatorics}
\usepackage[initials]{amsrefs}
\usepackage{verbatim}
\usepackage{xspace}

\newtheorem{theorem}{Theorem}%[section]
\newtheorem*{theorem*}{Theorem}

\newtheorem{lemma}[theorem]{Lemma}

\theoremstyle{definition}
\newtheorem{defn}[theorem]{Definition}
\theoremstyle{remark}

\newtheorem{remark}[theorem]{Remark}

\numberwithin{equation}{section}

\newcommand{\N}{{\mathbf N}}
\newcommand{\R}{{\mathbf R}}

\newcommand{\Z}{{\mathbf Z}}
\newcommand{\Q}{{\mathbf Q}}

\newcommand{\sub}{\subseteq}
\newcommand{\sminus}{\setminus}

\newcommand{\boplus}{\textstyle\bigoplus}

\newcommand{\set}[1]{\{#1\}}
\newcommand{\genby}[1]{\langle#1\rangle}
\newcommand{\nm}[1]{\left\| #1 \right\|}

\DeclareMathOperator{\len}{len}
\DeclareMathOperator{\vol}{vol}
\DeclareMathOperator{\Span}{span}
\DeclareMathOperator{\rank}{rank}
\DeclareMathOperator{\ord}{ord}

\newcommand{\df}{\emph}

\newcommand{\romanenum}{\renewcommand{\theenumi}{\roman{enumi}}}
\renewcommand{\And}{\text{ and }}

\begin{document}

\begin{abstract}
We prove that the additive group of the rationals does not have an automatic presentation. The proof also applies to certain other abelian groups, for example, torsion-free groups that are $p$-divisible for infinitely many primes $p$, or groups of the form $\boplus_{p \in I} \Z(p^\infty)$, where $I$ is an infinite set of primes.
\end{abstract}

\maketitle

\section{Introduction}
Consider the basic algorithm for adding two integers that we are taught in elementary school: write the two numbers in decimal, align them on the right, and add them digit by digit (using an addition table), carrying only one bit of information from one position to the next. What is remarkable about this procedure is that we can add two very long integers, digit by digit, using only local information and a bounded amount of memory. It becomes interesting to understand what other mathematical structures admit an encoding such that one can perform the operations using a similarly simple algorithm. This idea is formalized by the notion of an \df{automatic} (or \df{FA-presentable}) structure which is defined as follows.

Fix a finite alphabet $\Sigma$ and denote by $\Sigma^*$ the set of all finite words formed by letters of $\Sigma$. A \df{language} is a subset of $\Sigma^*$. A language is called \df{regular} if there exists a finite automaton that recognizes it. The following definition was first considered by Hodgson~\cite{Hodgson1982} and the basic theory of automatic structures was later developed by Khoussainov and Nerode~\cite{Khoussainov1995}.
\begin{defn} \label{d:autom}
A countable, relational structure $(M; R_1, \ldots, R_k)$, where $M$ is the universe of the structure and $R_1, \ldots, R_k$ are the relations, is called \df{automatic} if there exists a regular language $D \sub \Sigma^*$ and a bijection $g \colon D \to M$ such that the relations $g^{-1}(R_1), \ldots, g^{-1}(R_k)$ are also regular.
\end{defn}
In order to make sense of what it means for $g^{-1}(R_i)$ to be regular, one has to specify how to represent $(\Sigma^*)^n$ as a set of words in a finite alphabet. The standard approach is to use \df{padding}: add a special symbol $\diamond$ to the alphabet and embed $(\Sigma^*)^n$ into $((\Sigma \cup \set{\diamond})^n)^*$ by appending $\diamond$s at the end of the shorter words in the $n$-tuple so that all words become of equal length. Everywhere below where we mention regular subsets of $(\Sigma^*)^n$, we are using this convention. For more details, see any of the papers \cites{Khoussainov1995, Khoussainov2007, Khoussainov2008p}. In Definition~\ref{d:autom}, one can relax the condition on $g$ and allow it to be only a surjection but then equality in $M$ has to be regular. Also, one can include in the definition structures with function symbols by considering the graphs of the functions as relations. This will be important for us because we will be mostly concerned with algebraic structures.

Automatic structures are also attractive from another point of view: since the class of regular languages is stable under Boolean operations and projections, one readily sees that for any first order formula $\phi(\bar{x})$, the set $\set{\bar{a} \in D^n : D \models \phi(\bar{a})}$ is a regular language and, moreover, one can construct algorithmically an automaton recognizing it starting from the formula $\phi$ and the automata for the basic relations. In particular, the first order theories of automatic structures are decidable. One can also extend the first order language by the additional quantifiers ``there exist infinitely many'' and ``there exist $m$ modulo $n$'' and keep this decidability property. For all of this and some additional background, see Rubin's thesis \cite{Rubin2004} or the recent survey Khoussainov--Minnes~\cite{Khoussainov2008p}.

The condition of admitting an automatic presentation turns out to be rather restrictive. If one allows rich algebraic structure in the language, then often the only automatic structures are the trivial ones. For example, every automatic Boolean algebra is either finite or a finite power of the algebra of finite and co-finite subsets of $\N$ and all automatic integral domains (in the language of rings) are finite (Khoussainov--Nies--Rubin--Stephan~\cite{Khoussainov2007}; for more detailed information on automatic rings, see also Nies--Thomas~\cite{Nies2008}).

Even if one considers simpler algebraic structures such as groups, the definition is still too restrictive: Oliver and Thomas~\cite{Oliver2005} observed, as a consequence of Gromov's theorem about finitely generated groups of polynomial growth and a theorem of Romanovski{\u \i} classifying the virtually polycyclic groups with decidable first order theory, that a finitely generated group has an automatic presentation (in the sense of Definition~\ref{d:autom}) iff it is virtually abelian. This was extended by Nies and Thomas~\cite{Nies2008} who showed that every finitely generated subgroup of an automatically presentable group is virtually abelian.

However, for finitely generated groups, there is a convenient alternative. A different notion of an \df{automatic group}, in which the alphabet is a set of generators for the group, each word represents the corresponding product of generators, and one further requires that equality in the group and right multiplication by a generator be verifiable by automata, was introduced by Cannon and Thurston in the 1980s (see Epstein et al.~\cite{Epstein1992} for the precise definition and more details) and has led to a rich and interesting theory. In order to avoid confusion, we will adopt the terminology from \cite{Nies2008} and call a group with an automatic presentation in the sense of Definition~\ref{d:autom} \df{FA-presentable}.

In view of the remarks above, it seems that the natural class of groups for which one wants to consider FA-presentability is the class of abelian groups and this is where we will concentrate our attention from now on. There are already some interesting known examples. Finite groups are of course FA-presentable and an infinite direct sum of copies of $\Z/p\Z$ is also FA-presentable. By using the idea of ``addition with carry,'' one can construct presentations for $\Z$ and $\Z(p^\infty) = \set{x \in \Q / \Z : \exists k\ p^kx = 0}$. The class of FA-presentable groups is stable under finite sums (so all finitely generated abelian groups are FA-presentable) and one can combine a presentation of $\Z$ with a presentation of $\boplus_{p \mid n} \Z(p^\infty)$ to construct a presentation of $\Z[1/n] = \set{a/n^k \in \Q : a, k \in \Z}$. The class of FA-presentable abelian groups is also stable under taking finite extensions and, more interestingly, under ``automatic amalgamation'' (Nies--Semukhin~\cite{Nies2007a}) which provides some further examples. Currently, there are fairly few known ways to show that an abelian group does not admit an automatic presentation: the only abelian groups with a decidable first order theory known to not be FA-presentable were the ones containing a free abelian group of infinite rank~\cite{Khoussainov2007}. In this paper, we describe some new restrictions on possible automatic presentations of abelian groups. The following is our main theorem which answers a question of Khoussainov (see, e.g., \cite{Khoussainov2008}).
\begin{theorem} \label{th:main}
The following groups are not FA-presentable:
\begin{enumerate} \romanenum
\item $(\Q, +)$, or, more generally, any torsion-free abelian group that is $p$-divisible for infinitely many primes $p$; \label{eq:main:div}
\item $(\Q / \Z, +)$, or, more generally, any group of the form $\boplus_{p \in I} \Z(p^\infty)$, where $I$ is an infinite set of primes. \label{eq:main:QZ}
\end{enumerate}
\end{theorem}
Some partial results providing restrictions on possible automatic presentations of $\Q$ and $\Q / \Z$ had been proved by F.~Stephan (see \cite{Nies2007}).

The ideas for the proof of Theorem~\ref{th:main} are combinatorial. Our main tool is Freiman's structure theorem for sets with a small doubling constant.

The organization of the paper is as follows. In Section~\ref{s:add-c}, we discuss some preliminary notions and facts from additive combinatorics; in Section~\ref{s:proof}, we prove Theorem~\ref{th:main} for the case of the rationals; and finally, in Section~\ref{s:other}, we indicate how to modify the proof in order to obtain the other instances of Theorem~\ref{th:main}.

Below, $\N$, $\Z$, $\Q$, and $\R$ will denote the sets of the natural numbers, the integers, the rationals, and the reals, respectively. If $A$ is a finite set, $|A|$ denotes its cardinality.

{\bf Acknowledgements. } I am grateful to B.~Khoussainov for pointing out an error in a preliminary draft of this paper, making many useful comments, and suggesting some references.

\section{Preliminaries from additive combinatorics} \label{s:add-c}
Our main reference for results in additive combinatorics is the book by Tao and Vu~\cite{Tao2006}.

Let $Z$ be an abelian group. We will be interested in finite sets $A \sub Z$ such that their doubling $A+A = \set{a_1 + a_2 : a_1, a_2 \in A}$ is small, i.e., $|A+A| \leq C |A|$ for some constant $C$ (such sets naturally arise from automatic presentations of $Z$ as we shall see shortly). Typical sets with this property are arithmetic progressions and, more generally, multi-dimensional arithmetic progressions. By a remarkable theorem of Freiman~\cite{Freiman1966}, these are essentially the only examples. In order to state the theorem, we recall a few basic definitions. A \df{generalized arithmetic progression} (or just a progression, for short) in an abelian group $Z$ is a pair $(P, \phi)$, where $P$ is a finite subset of $Z$ and $\phi$ is an affine map from a parallelepiped in $\Z^d$ onto $P$, i.e.,
\[ P = \set{ v_0 + \sum_{i=1}^d a_i v_i : 0 \leq a_i < N_i \text{ for } i = 1, \ldots, d}, \]
where $v_0, v_1, \ldots, v_d \in Z$ and $N_1, \ldots, N_d \in \N$ (and of course, $\phi(a_1, \ldots, a_d) = v_0 + \sum_{i=1}^d a_i v_i$). We will often suppress $\phi$ if it is clear from the context. The number $d$ is called the \df{rank} of the progression. Progressions of rank $1$ are just ordinary arithmetic progressions. A progression is called \df{proper} if $\phi$ is injective. Also, if $N = (N_1, \ldots, N_d)$, we will write $[0, N)$ for the parallelepiped $\prod_{i=1}^d [0, N_i)$ in $\Z^d$ (and similarly for $(-N, N)$, etc.). If we put $v = (v_1, \ldots, v_d)$ and $a = (a_1, \ldots, a_d) \in \Z^d$, then we will write $a \cdot v$ for the sum $\sum_{i=1}^d a_i v_i \in Z$. With this notation, we can concisely write the progression $P$ as $v_0 + [0, N) \cdot v$.

\begin{theorem*}[Freiman's theorem]
Let $Z$ be a torsion-free abelian group and $C > 0$ be a constant. Then there exist constants $K$ and $d$ such that whenever a finite set $A \sub Z$ satisfies $|A + A| \leq C |A|$, there exists a proper progression $P$ of rank at most $d$ that contains $A$ and $|P| / |A| \leq K$.
\end{theorem*}

The original proof of the theorem can be found in \cite{Freiman1966}; for a modern treatment due to Ruzsa, see Ruzsa~\cite{Ruzsa1994}, Tao--Vu~\cite{Tao2006}*{Chapter~5}, or the self-contained exposition Green~\cite{Green2002p}.

We will also need some basic notions and facts from the geometry of numbers. Recall that a \df{lattice} in $\R^d$ is a discrete subgroup. The \df{rank} of a lattice is the dimension of the subspace of $\R^d$ that it spans. A subset $B \sub \R^d$ is \df{symmetric} if $B = -B$. We denote by $\vol$ the $d$-dimensional Lebesgue measure. The following lemma goes back to Minkowski and follows for example from \cite{Tao2006}*{Theorem~3.30}; in order to avoid introducing additional notation, we supply the easy proof.
\begin{lemma} \label{l:small-conv-set}
Let $B \sub \R^d$ be an open, symmetric, convex set and $\Gamma < \R^d$ be a lattice of full rank. If $\vol(B) < (2^d / d!) \vol(\R^d / \Gamma)$, then $\dim \Span B \cap \Gamma < d$.
\end{lemma}
\begin{proof}
Suppose, towards contradiction, that $B \cap \Gamma$ contains $d$ linearly independent vectors $v_1, \ldots, v_d$. By applying an invertible linear transformation of $\R^d$ (which will scale both sides of the given inequality by the same factor), we can assume that $(v_1, \ldots, v_d)$ is in fact the standard basis of $\R^d$. In particular, after this transformation, $\Gamma$ will contain $\Z^d$ and hence, $\vol(\R^d / \Gamma) \leq 1$. On the other hand, $B$, being convex and symmetric, will contain the polyhedron with vertices $\pm v_1, \ldots, \pm v_d$ which has volume $2^d / d!$. This contradicts the hypothesis.
\end{proof}

One last fact which we will need is that the intersection of a convex set with a lattice can be efficiently contained in a progression of rank equal to the rank of the lattice. More precisely, the following holds (see \cite{Tao2006}*{Lemma~3.36}).
\begin{lemma} \label{l:discrete-John}
Let $B$ be a convex, symmetric, open set in $\R^d$ and let $\Gamma < \R^d$ be a lattice of rank $r$. Then there exist a tuple $w= (w_1, \ldots, w_r) \in \Gamma^r$ of linearly independent vectors in $\R^d$ and a tuple $N = (N_1, \ldots, N_r)$ of positive integers such that
\[ (-N, N) \cdot w \sub B \cap \Gamma \sub (-r^{2r} N, r^{2r} N) \cdot w. \]
\end{lemma}

\section{Proof for the case of the rationals} \label{s:proof}
Let $\Sigma$ be a finite alphabet. If $L \sub \Sigma^*$, denote by $L^{\leq n}$ the set of words in $L$ of length not greater than $n$. We will need the following two basic lemmas (for proofs, see, for example, \cite{Khoussainov2007}). The first one is a general fact about the growth of regular languages and the second is a version of the pumping lemma particularly suitable for studying automatic structures.

\begin{lemma} \label{l:exp-growth}
Let $L \sub \Sigma^*$ be a regular language. Then there exists a constant $C$ such that $|L^{\leq n+1}| \leq C |L^{\leq n}|$ for all $n$.
\end{lemma}

\begin{lemma} \label{l:finite-sections}
Let $L_1, L_2$ be languages over a finite alphabet and $R \sub L_1 \times L_2$ be a regular relation such that the sections $R_x = \set{y \in L_2 : (x, y) \in R}$ are finite. Then for all $(x, y) \in R$, $\len(y) \leq \len(x) + k$, where $k$ is the number of states of an automaton recognizing $R$.
\end{lemma}

Suppose now that $(Z, +)$ is an FA-presentable abelian group and fix some automatic presentation of it; that is, fix a regular language $D \sub \Sigma^*$ and a bijection $g \colon D \to Z$ such that the preimage under $g$ of the graph of addition is recognizable by an automaton with, say, $r$ states. We will often identify $Z$ and $D$ via $g$. For example, when we write $A + B$ for some $A, B \sub D$, we mean the set $\set{g^{-1}(g(a) + g(b)) : a \in A, b \in B}$. By applying Lemma~\ref{l:finite-sections} to the graph of addition, one immediately obtains that $D^{\leq n} + D^{\leq n} \sub D^{\leq n+r}$. Also, the graph of the homomorphism $M_p \colon Z \to Z$ defined by $M_p(x) = px$, where $p$ is an integer, is regular. Let $h(p)$ be the minimal number of states of an automaton recognizing the graph of $M_p$. (Using the fact that one can compute $M_p(x)$ using no more than $O(\log p)$ additions, one sees that $h(p) = p^{O(1)}$ but we will not need this.) If $A \sub Z$, denote by $p^{-1}A$ the set $M_p^{-1}(A)$. If $M_p$ has finite kernel (for example, if $Z$ is torsion-free), Lemma~\ref{l:finite-sections} implies that $p^{-1} D^{\leq n} \sub D^{\leq n + h(p)}$.

Let $l_0 = \min \set{l \in \N : 0 \in D^{\leq l} \And |D^{\leq l}| \geq 2}$ and put $A_n = D^{\leq l_0 + nr}$ for $n = 0, 1, \ldots$. We summarize the properties of the sets $A_n$ that we have established so far (under the assumption that the homomorphism $M_p$ has a finite kernel). There exist a constant $C_1$ and a function $h \colon \N \to \N$ such that:
\begin{enumerate} \romanenum
\item $0 \in A_0$ and $|A_0| \geq 2$; \label{i:A0}
\item $A_n + A_n \sub A_{n+1}$; \label{i:double}
\item $|A_{n+1}| \leq C_1 |A_n|$; \label{i:growth}
\item $p^{-1} A_n \sub A_{n + h(p)}$. \label{i:hp}
\end{enumerate}
The property \eqref{i:growth} follows from Lemma~\ref{l:exp-growth}. In particular \eqref{i:double} and \eqref{i:growth} imply that
\begin{equation} \label{eq:small-doubling}
|A_n + A_n| \leq C_1 |A_n| \quad \text{for all } n.
\end{equation}

Now we can formulate our main combinatorial result which, by the above observations, implies Theorem~\ref{th:main} for the case of the rationals.
\begin{theorem} \label{th:main-comb}
There does not exist a sequence $\set{A_n}_{n \in \N}$ of finite subsets of $\Q$ that satisfy the conditions \eqref{i:A0}--\eqref{i:hp} above.
\end{theorem}
\begin{proof}
We will obtain a contradiction with Freiman's theorem. To that end, we will need a quantitative measure of how efficiently a given additive set is contained in a progression. For a finite additive set $A$ and a rank $d$, define
\[ \theta(A, d) = \min \set{ |P|/|A| : P \supseteq A \text{ and } P \text{ is a proper progression of rank } \leq d }. \]
If there is no $d$-dimensional progression covering $A$ (for example, if $d = 0$), put $\theta(A, d) = \infty$. One property of $\theta$, obvious from the definition, is the following:
\begin{equation} \label{eq:theta-supset}
B \sub A \implies \theta(A, d) \geq \frac{|B|}{|A|} \theta(B, d).
\end{equation}
The next lemma quantitatively formalizes the observation that if we have a progression of integers all of whose elements are divisible by $p$ for some large prime $p$ and add to it a single element not divisible by $p$, then in order to contain the resulting set efficiently in a progression, we need to increase its rank. In order to state the lemma in a slightly more general form that will be useful later, we introduce some notation. If $A$ is a subset of an abelian group, denote by $\genby{A}$ the group generated by $A$.

Recall that if $p$ is a prime, the \df{$p$-adic norm} $\nm{x}_p$ of $x \in \Q \sminus \set{0}$ is defined by
\[ \nm{x}_p = p^m \iff x = p^{-m} a/b, \text{ where } a, b \in \Z \sminus p\Z \And m \in \Z, \]
and $\nm{0}_p = 0$. Note that the $p$-adic norm on $\Q$ has the following properties:
\begin{gather}
\set{\nm{x}_p : x \in \Q} \text{ has no accumulation points other than }0
    \And \forall x \ \lim_{m \to \infty} \nm{p^m x}_p = 0; \label{eq:um1:discrete} \\
\nm{x}_p = \nm{-x}_p \And \nm{x+y}_p \leq \max \set{\nm{x}_p, \nm{y}_p}; \label{eq:um1:ultra} \\
\forall a \in \Z \quad \nm{ax}_p < \nm{x}_p \implies p \mid a. \label{eq:um1:div}
\end{gather}
For a set $A \sub \Q$, let $\nm{A}_p = \sup \set{\nm{a}_p : a \in A}$. If $V \leq \Q$ and $0 < \nm{V}_p < \infty$, let $V_{(p)}$ denote the subgroup $\set{x \in V : \nm{x}_p < \nm{V}_p}$. Note that \eqref{eq:um1:discrete}--\eqref{eq:um1:div} imply the following:
\begin{gather}
\text{ for all } A \sub \Q, \quad \nm{\genby{A}}_p = \nm{A}_p; \\
\nm{x}_p > \nm{y}_p \implies \nm{x+y}_p = \nm{x}_p; \label{eq:um2:pres} \\
0 < \nm{V}_p < \infty \implies [V : V_{(p)}] \geq p. \label{eq:um2:index}
\end{gather}
To see that \eqref{eq:um2:index} holds, note that if $v \in V$ is such that $\nm{v}_p = \nm{V}_p$ (which exists by \eqref{eq:um1:discrete}), then, by \eqref{eq:um1:div}, the elements $0, v, 2v, \ldots, (p-1)v$ of $V$ are in distinct cosets of the subgroup $V_{(p)}$.
\begin{lemma} \label{l:incr-rank}
Let $d \geq 1$ be an integer and $p > d!$ be prime. Let $Z$ be an abelian group equipped with a norm $\nm{\cdot}_p$ satisfying \eqref{eq:um1:discrete}--\eqref{eq:um1:div}. Let $A \sub Z$ be a finite set with at least $2$ elements, and $z \in Z$ be such that $\nm{z}_p > \nm{A}_p$. Then
\begin{equation} \label{eq:l:incr-rank}
 \theta(A \cup \set{z}, d) \geq \min \set{\frac{p^{1/d}}{4d}, \frac{\theta(A, d-1)}{d^{C_0 d^3}} },
\end{equation}
where $C_0$ is an absolute constant.
\end{lemma}
\begin{proof}
If $\theta(A \cup \set{z}, d) = \infty$, there is nothing to prove, so suppose that $\theta(A \cup \set{z}, d) < \infty$. Let
\[ P = \set{ v_0 + \sum_{i=1}^d a_i v_i : 0 \leq a_i < N_i \text{ for } i = 1, \ldots, d}, \]
where $v_0, \ldots, v_d \in Z$, be a proper progression of rank $d$ covering $A \cup \set{z}$ such that $|P| / |A \cup \set{z}| = \theta(A \cup \set{z}, d)$. We will first show that for some $i \geq 1$, $\nm{v_i}_p \geq \nm{z}_p$. Denote by $V_1$ the group generated by $v_1, \ldots, v_d$ and suppose, towards a contradiction, that $\nm{V_1}_p < \nm{z}_p$. We have two cases: either $\nm{v_0}_p < \nm{z}_p$ or $\nm{v_0}_p \geq \nm{z}_p > \nm{V_1}_p$. In the first case, we obtain that, since $P \sub v_0 + V_1$, $\nm{P}_p \leq \nm{v_0 + V_1}_p < \nm{z}_p$ which contradicts the fact that $z \in P$. In the second, by \eqref{eq:um2:pres}, for every $x \in v_0 + V_1$, $\nm{x}_p = \nm{v_0}_p > \nm{A}_p$ which contradicts the fact that $A \sub v_0 + V_1$ and $A$ is non-empty.

Now by reordering $v_1, \ldots, v_d$, we can assume that there exists $k \geq 1$ such that
\[ \nm{v_1}_p, \ldots, \nm{v_k}_p \geq \nm{z}_p \And \nm{v_{k+1}}_p, \ldots, \nm{v_d}_p < \nm{z}_p. \]
Put $M = N_1 N_2 \cdots N_k$. We distinguish the following two cases which will correspond to the two different quantities on the right-hand side of \eqref{eq:l:incr-rank}.

{\bf Case 1.} $M < p/k!$. Note that this case is impossible if $d = 1$. Indeed, we showed that $\nm{v_1}_p > \nm{A}_p$ and since $A$ contains at least two elements, it is easy to see, using \eqref{eq:um1:div}, that if $d = 1$, the interval $[0, N_1)$ must contain two integers whose difference is divisible by $p$ and hence, $M = N_1 \geq p+1$ (for a similar argument, see Case 2 below). Hence, we can assume that $d \geq 2$.

Write $v$ for the vector $(v_1, \ldots, v_k)$ and $N$ for $(N_1, \ldots, N_k)$. Let $f \colon \Z^k \to Z$ be the homomorphism $f(x) = x \cdot v$. Put $V = \genby{v_1, \ldots, v_k} = f(\Z^k)$. Let
\[ \Lambda = \set{y \in Z : \nm{y}_p < \nm{z}_p} \]
and note that $\Lambda$ is a subgroup of $Z$ and $A \sub \Lambda$. Let also $\Gamma$ be the lattice in $\R^k$ given by
\[ \Gamma = \set{x \in \Z^k : f(x) \in \Lambda} \]
and note that since by \eqref{eq:um1:discrete}, for all large enough $m$, $p^m \Z^k \leq \Gamma$, $\Gamma$ has full rank. Let $B$ be the open, symmetric, convex box $\prod_{i=1}^k (-N_i, N_i)$ in $\R^k$. We have $\vol(B) = (2N_1) \cdots (2N_k) = 2^k M$ and
\[ \vol(\R^k / \Gamma) = \vol(\R^k / \Z^k) [\Z^k : \Gamma] \geq p. \]
The last inequality follows from the fact that $\Gamma$ is contained in the kernel of the composition of the surjective homomorphisms
\[ \Z^k \xrightarrow{f} V \to V/V_{(p)} \]
and \eqref{eq:um2:index}. Applying Lemma~\ref{l:small-conv-set} and our hypothesis about $M$ yields that $\Gamma \cap B$ is contained in a sublattice of $\Gamma$ of rank $r < k$. (For the moment, suppose that $k > 1$, so that we can take $r > 0$. We will explain how to deal with the case $k=1$ later.) By Lemma~\ref{l:discrete-John}, there exist tuples $N' = (N_1', \ldots, N_r')$ of positive integers and $w = (w_1, \ldots, w_r) \in \Gamma^r$ such that $w_1, \ldots, w_r$ are independent in $\R^k$ and
\[ (-N', N') \cdot w \sub \Gamma \cap B \sub (-r^{2r}N', r^{2r}N') \cdot w. \]
From the first inclusion and the independence of $w_1, \ldots, w_r$, we have
\begin{equation} \label{eq:est-N'}
   |(-N', N')| = |(-N', N') \cdot w| \leq |B \cap \Gamma| \leq
    |B \cap \Z^k| < 2^k M.
\end{equation}
Let $P_0$ be the progression $\set{\sum_{i=k+1}^d a_i v_i : 0 \leq a_i < N_i}$ and note that $P_0 \sub \Lambda$. Then $P = v_0 + P_0 + f([0, N))$. Note that by the properness of $P$, $|P| = |P_0| \cdot |[0, N)| = M|P_0|$. Let $a^0 = (a_1^0, \ldots, a_k^0) \in [0, N)$ be such that $v_0 + a^0 \cdot v \in \Lambda$ and put $v_0' = a^0 \cdot v$ (since $A \sub P \cap \Lambda$, such an $a^0$ always exists). We have
\[ P = v_0 + P_0 + f([0, N)) = v_0 + P_0 + v_0' + f([-a^0, N-a^0)). \]
Note that $v_0 + v_0' + P_0 \sub \Lambda$. Hence,
\begin{align}  \label{eq:new-progr}
A \sub P \cap \Lambda &= v_0 + v_0' + P_0 + f([-a^0, N-a^0)) \cap \Lambda \notag \\
    &= v_0 + v_0' + P_0 + f([-a^0, N-a^0) \cap \Gamma) \notag \\
    &\sub v_0 + v_0' + P_0 + f(B \cap \Gamma) \notag \\
    &\sub v_0 + v_0' + P_0 + f((-r^{2r}N', r^{2r}N') \cdot w)\notag \\
    &\sub v_0 + v_0' + P_0 + (-r^{2r}N', r^{2r}N') \cdot f(w),
\end{align}
where $f(w) = ((f(w_1), \ldots, f(w_r))$. Denote by $Q$ the progression \eqref{eq:new-progr}. We have that $A \sub Q$, $\rank Q = \rank P_0 + r = d - k + r < d$, and by \eqref{eq:est-N'},
\begin{equation} \label{eq:sizeQ}
 \begin{split}
|Q| &\leq |P_0| \cdot |(-r^{2r}N', r^{2r}N')| \\
  &< |P_0| 2^r r^{2r^2} |(-N', N')| \\
  &< |P_0| 2^r r^{2r^2} 2^k M < 4^k k^{2k^2} |P|.
\end{split}
\end{equation}
Now note that if we had $k = 1$ in the beginning, then $B \cap \Gamma = \set{0}$, so if we take $Q = v_0 + v_0' + P_0$, we will again have $A \sub Q$, $\rank Q < d$, and the estimate \eqref{eq:sizeQ} will still hold.

Of course, the progression $Q$ need not be proper. However, properness can be achieved at the price of increasing its size. Applying \cite{Tao2006}*{Theorem~3.40} yields that we can include $Q$ in a proper progression $Q'$ of equal or lesser rank and size at most $d^{C_0' d^3} |Q|$ for some absolute constant $C_0'$. This allows us to conclude that in this case,
\[ \begin{split}
\theta(A \cup \set{z}, d) &= \frac{|P|}{|A \cup \set{z}|} \geq \frac{|Q|}{2|A| 4^k k^{2k^2}} \\
  &\geq \frac{|Q'|}{2|A|d^{C_0' d^3}4^k k^{2k^2}} \\
  &\geq \frac{\theta(A, d-1)}{d^{C_0 d^3}}
\end{split} \]
for an appropriately chosen $C_0$.

{\bf Case 2.} $M \geq p/k!$. Then for some $i \leq k$, $N_i \geq (p/k!)^{1/k} \geq (p/d!)^{1/d}$. Without loss of generality, we can assume that $N_1 \geq (p/d!)^{1/d}$. Now fix some $(a_2, \ldots, a_d) \in \Z^{d-1}$ and consider the following condition on $a_1$:
\begin{equation} \label{eq:result-in-pZ}
a_1 v_1 + v_0 + \sum_{i=2}^d a_i v_i \in \Lambda.
\end{equation}
Let $a_1', a_1'' \in \Z$ be two values of $a_1$ satisfying \eqref{eq:result-in-pZ}. Since $\nm{v_1}_p > \nm{\Lambda}_p$, by \eqref{eq:um1:div}, we obtain that $p \mid a_1' - a_2''$. Hence the proportion of the numbers $a_1$ in the interval $[0, N_1)$ for which \eqref{eq:result-in-pZ} holds is not greater than $\lceil N_1 / p \rceil / N_1 \leq \max \set{1/N_1, 2/(p+1)}$. Therefore, by properness,
\[ |P \cap \Lambda|/|P| \leq \max \set{(p/d!)^{-1/d}, 2/(p+1)} \leq (p/(2d!))^{-1/d}. \]
Hence in this case,
\[ \theta(A \cup \set{z}, d) = \frac{|P|}{|A \cup \set{z}|} \geq \frac{|P|}{2|A|} \geq \frac{|P|}{2|P \cap \Lambda|} \geq \frac{(p/(2d!))^{1/d}}{2} \geq \frac{p^{1/d}}{4d}. \qedhere \]
\end{proof}

Now we can proceed with the proof of the theorem. Suppose, towards a contradiction, that a sequence of subsets $\set{A_n}$ of $\Q$ satisfying \eqref{i:A0}--\eqref{i:hp} does exist. Let $C = \max \set{C_0, C_1}$, where $C_0$ is the constant from Lemma~\ref{l:incr-rank} and $C_1$ is the constant from \eqref{i:growth}. By \eqref{eq:small-doubling} and Freiman's theorem, there exist constants $K$ and $d$ such that $\theta(A_n, d) \leq K$ for all $n$. Pick inductively a sequence of primes $p_{d} < p_{d-1} < \cdots < p_0$ satisfying the conditions
\begin{equation} \label{eq:choice-primes}
 p_{d} > C(4dK)^d \quad \text{and}
    \quad p_{i-1} > p_i C^{h(p_i)d} d^{Cd^4} \text{ for } i = d, d-1, \ldots, 1.
\end{equation}
Define inductively the sequence of integers $n_0 < n_1 < \cdots < n_{d}$ by
\[ n_0 = 0 \quad \text{and} \quad n_i = \min \set{n : \nm{A_n}_{p_i} > \nm{A_{n_{i-1}}}_{p_i}}
    \quad \text{for } i = 1, \ldots, d. \]
Note that by the properties of the family $\set{A_n}$, $n_i \leq n_{i-1} + h(p_i)$ (indeed, if the norm $\nm{A_{n_{i-1}}}_{p_i}$ is achieved for $z \in A_{n_{i-1}}$, then $\nm{p_i^{-1}z}_{p_i} > \nm{z}_{p_i} = \nm{A_{n_{i-1}}}_{p_i}$ and $p_i^{-1}z \in A_{n_{i-1} + h(p_i)}$). Hence,
\begin{equation} \label{eq:An-est}
|A_{n_i-1}| \leq C^{h(p_i)+1} |A_{n_{i-1}}|.
\end{equation}

We will prove by induction on $i$ that
\begin{equation} \label{eq:induction}
\theta(A_{n_i}, i) > C^{-1} p_i^{1/d}/(4d) \quad \text{for all } i = 0, \ldots, d.
\end{equation}
Applied for $i = d$, \eqref{eq:induction} will yield a contradiction with the choice of $p_{d}$. The case $i = 0$ follows trivially from the definition of $\theta$. Suppose now that $i \geq 1$ and \eqref{eq:induction} holds for $i-1$ in order to prove it for $i$. By the induction hypothesis, \eqref{eq:An-est}, and \eqref{eq:theta-supset},
\begin{equation} \label{eq:Ani-1est}
 \theta(A_{n_i-1}, i-1) \geq C^{-h(p_i)+1} \theta(A_{n_{i-1}}, i-1) > C^{-h(p_i)} p_{i-1}^{1/d}/(4d).
\end{equation}
By the choice of $n_i$, there exists $z \in A_{n_i}$ such that $\nm{z}_{p_i} > \nm{A_{n_i-1}}_{p_i}$. Apply Lemma~\ref{l:incr-rank} to the set $A_{n_i-1} \cup \set{z}$ and the prime $p_i$ to obtain
\[ \begin{split}
\theta(A_{n_i}, i) &> C^{-1} \theta(A_{n_i-1} \cup \set{z}, i) \\
  &\geq C^{-1} \min \set{ p_i^{1/d}/(4d), \theta(A_{n_i-1}, i-1)/d^{Cd^3} }.
\end{split} \]
The choice of $p_{i-1}$ and \eqref{eq:Ani-1est} allow us to conclude that $p_i^{1/d}/(4d) \leq \theta(A_{n_i-1}, i-1)/d^{Cd^3}$ which completes the induction and the proof.
\end{proof}

\begin{remark} \label{rm:inf-rank}
Note that Freiman's theorem gives another way to see that a torsion-free abelian group of infinite rank is not FA-presentable (originally proved in \cite{Khoussainov2007}). Indeed, if one considers the sets $D^{\leq n}$ as above, by applying Freiman's theorem, one obtains that there is a constant $d$ such that each $D^{\leq n}$ is contained in a progression of rank $d$. Since the group generated by a progression of rank $d$ has rank at most $d+1$, this leads to a contradiction. In fact, for this argument, instead of Freiman's theorem, one can use the much simpler Freiman lemma \cite{Tao2006}*{Lemma~5.13}.
\end{remark}

\section{Other groups} \label{s:other}
\subsection{The torsion-free case}
The proof above can be used to show that certain other torsion-free abelian groups also do not have an automatic presentation. Recall that an abelian group $Z$ is called \df{$p$-divisible} if all of its elements are divisible by $p$, i.e., for all $x \in Z$, there exists $y \in Z$ such that $p y = x$. It is easy to extend the proof in Section~\ref{s:proof} to cover all torsion-free groups that are $p$-divisible for infinitely many $p$. Let $Z$ be such a group. If $Z$ has infinite rank, then $Z$ is not FA-presentable by \cite{Khoussainov2007} (cf. Remark~\ref{rm:inf-rank} above). Otherwise, $Z$ can be embedded as a subgroup of $\Q^k$ for some finite $k$. For $x = (x_1, \ldots, x_k) \in \Q^k$, define its $p$-adic norm by
\[ \nm{x}_p = \max \set{\nm{x_1}_p, \ldots, \nm{x_k}_p}. \]
It is easy to check that this norm satisfies \eqref{eq:um1:discrete}--\eqref{eq:um1:div}, hence Lemma~\ref{l:incr-rank} applies. In order to complete the rest of the proof of Theorem~\ref{th:main-comb}, one just has to choose the primes $p_1, \ldots, p_d$ in \eqref{eq:choice-primes} so that $Z$ is $p_i$-divisible for each $i$. That can be done because, by assumption, there are infinitely many such primes. This completes the proof of Theorem~\ref{th:main} \eqref{eq:main:div}.

\subsection{The torsion case}
One has to be slightly more careful in the torsion case but the proof in Section~\ref{s:proof} still goes through for some torsion groups. Let $I$ be some infinite set of primes and put $T_I = \boplus_{p \in I} \Z(p^\infty)$. (In the special case when $I$ is the set of all primes, $T_I = \Q / \Z$.) For $p \in I$, one can define the $p$-adic norm for $x \in T_I$ by
\[ \nm{x}_p = \ord \pi_p(x), \]
where $\pi_p \colon T_I \to \Z(p^\infty)$ is the natural projection and $\ord z$ denotes the order of $z$ (with the special agreement that $\ord 0 = 0$). This is not really a norm (in the sense that $\set{x \in T_I : \nm{x}_p = 0}$ is a non-trivial subgroup of $T_I$) but it still satisfies \eqref{eq:um1:discrete}--\eqref{eq:um1:div} which is all we need for Lemma~\ref{l:incr-rank} to hold.

Freiman's theorem is also available for arbitrary abelian groups (Green--Ruzsa~\cite{Green2007}; see also \cite{Tao2006}*{Theorem~5.44}). The only difference is that now in the conclusion of the theorem, one obtains coset progressions instead of ordinary progressions. A \df{coset progression} in an abelian group $Z$ is a subset of the form $H + P$, where $H$ is a finite subgroup of $Z$, $P$ is a proper progression as defined previously, and the sum is direct, i.e., every element of $H + P$ can be represented in a unique fashion as a sum $h + p$, where $h \in H$ and $p \in P$. Since every finite subgroup of $T_I$ is cyclic and every finite cyclic group is a one-dimensional progression, every coset progression of rank $d$ in $T_I$ can be written as a proper progression of rank $d+1$.

The conditions \eqref{i:A0}--\eqref{i:hp} for the sets $A_n$ are still satisfied because the homomorphisms $M_p \colon T_I \to T_I$, $x \mapsto px$ have finite kernels for all primes $p$ . Also, one has to ensure that the primes $p_1, \ldots, p_d$ in \eqref{eq:choice-primes} are in the set $I$ which can be achieved because $I$ is infinite. The rest of the proof goes through unchanged.

\bibliography{mybiblio}
\end{document}